\documentclass{article}

\usepackage{amsfonts}
\usepackage{amssymb}
\usepackage{amsmath}
\usepackage{amstext}
\usepackage{graphicx}
\usepackage{latexsym}
\usepackage{amsthm}

\newtheorem{theorem}{Theorem}

\newtheorem{definition}[theorem]{Definition}
\newtheorem{example}[theorem]{Example}

\newtheorem{lemma}[theorem]{Lemma}

\newtheorem{remark}[theorem]{Remark}

\input tcilatex

\title{A generalization of Gauss' divergence theorem.}
\author{Vieri Benci, Lorenzo Luperi Baglini}

\begin{document}
\maketitle
\begin{abstract}
This paper is devoted to the proof Gauss' divergence theorem in the framework
of "ultrafunctions". They are a new kind of generalized functions, which
have been introduced recently in \cite{ultra} and developed in \cite{belu2012}, 
\cite{belu2013} and \cite{milano}. Their peculiarity is that they are based
on a non-Archimedean field, namely on a field which contains infinite and
infinitesimal numbers. Ultrafunctions have been introduced to provide
generalized solutions to equations which do not have any solutions, not even
among the distributions.
\end{abstract}

\section{Introduction}

In many problems of mathematical physics, the notion of function is not
sufficient and it is necessary to extend it. Among people working in partial
differential equations, the theory of distribution of Schwartz and the
notion of weak solution are the main tools to be used when equations do not
have classical solutions.

Usually, these equations do not have classical solutions since they develop
singularities. The notion of weak solutions allows to obtain existence
results, but uniqueness may be lost; also, these solutions might violate the
conservation laws. As an example let us consider the following scalar
conservation law:%
\begin{equation}
\frac{\partial u}{\partial t}+\func{div}F(t,x,u)=0,  \label{B}
\end{equation}%
where $F:\mathbb{R}_{t}\times \mathbb{R}_{x}^{N}\times \mathbb{R}%
_{u}\rightarrow \mathbb{R}_{x}^{N}$ satisfies the following assumption: $%
F(t,x,0)=0.$ A classical solution $u(t,x)$ is unique and, if it has compact
support, it preserves the quantity $Q=\int u\ dx.$ However, at some time a
singularity may appear and the phenomenon cannot be longer described by a
classical solution. The notion of weak solution becomes necessary, but the
problem of uniqueness becomes a central issue. Moreover, in general, $Q$ is
not preserved. From a technical point of view, the classical proof of
conservation of $Q$ fails since we cannot apply Gauss' divergence theorem to
weak solutions.

In this paper we suggest a method to overcome these problems. This method
consists in using a different kind of generalized solutions, namely
functions which belong to the space of "ultrafunctions". Ultrafunctions
have been introduced recently in \cite{ultra} and developed in \cite%
{belu2012}, \cite{belu2013}, \cite{milano}, \cite{algebra} and \cite{beyond}%
. The peculiarity of ultrafunctions is that they are based on a non-Archimedean field, namely a field which contains infinite and infinitesimal
numbers. The ultrafunctions have been introduced to provide generalized
solutions to equations which do not have any solutions, not even among distributions. However they provide also uniqueness in problems which have
more than one weak solution. Moreover, we will state a generalization of
Gauss' divergence theorem which can be applied to the study of partial
differential equations (see e.g. \cite{BLta}). Here we give a simple
application to equation (\ref{B}) using an elementary notion of generalized
solution (see section \ref{asa}). In a paper in preparation, we will give a
more appropriate notion of generalized solution of an evolution problem and
we will study in details the properties of the generalized solutions of
Burgers' equation.

\section{$\Lambda $-theory\label{lt}}

In this section we present the basic notions of non-Archimedean mathematics
and of nonstandard analysis following a method inspired by \cite{BDN2003}
(see also \cite{benci99}, \cite{ultra} and \cite{belu2012}).

\subsection{Non-Archimedean Fields\label{naf}}

Here, we recall the basic definitions and facts regarding non-Archimedean
fields. In the following, ${\mathbb{K}}$ will denote an ordered field. We
recall that such a field contains (a copy of) the rational numbers. Its
elements will be called numbers.

\begin{definition}
Let $\mathbb{K}$ be an ordered field. Let $\xi \in \mathbb{K}$. We say that:

\begin{itemize}
\item $\xi $ is infinitesimal if, for all positive $n\in \mathbb{N}$, $|\xi
|<\frac{1}{n}$;

\item $\xi $ is finite if there exists $n\in \mathbb{N}$ such that $|\xi |<n$%
;

\item $\xi $ is infinite if, for all $n\in \mathbb{N}$, $|\xi |>n$
(equivalently, if $\xi $ is not finite).
\end{itemize}

Moreover we let $x\sim y$ iff $|x-y|$ is infinitesimal. In this case we say that $x,y$ are infinitely close.
\end{definition}

Clearly, the relation "$\sim $" of infinite closeness is an equivalence
relation.

\begin{definition}
An ordered field $\mathbb{K}$ is called non-Archimedean if it contains an
infinitesimal $\xi \neq 0$.
\end{definition}

It is easily seen that all infinitesimal are finite, that the inverse of an
infinite number is a nonzero infinitesimal number and that the inverse of a
nonzero infinitesimal number is infinite.

\begin{definition}
A superreal field is an ordered field $\mathbb{K}$ that properly extends $\mathbb{R}$.
\end{definition}

It is easy to show, due to the completeness of $\mathbb{R}$, that there are
nonzero infinitesimal numbers and infinite numbers in any superreal field. Moreover, we have the following result on finite numbers.

\begin{theorem}
If $\mathbb{K}$ is a superreal field, every finite number $\xi \in \mathbb{K}
$ is infinitely close to a unique real number $r\sim \xi $.
\end{theorem}

\subsection{The $\Lambda $-limit\label{OL}}

In this section we will introduce a particular superreal field $\mathbb{K}$
and we will analyze its main properties by means of $\Lambda $-theory, in
particular by means of the notion of $\Lambda $-limit (for complete proofs
and for further properties of the $\Lambda $-limit, the reader is referred
to \cite{ultra}, \cite{belu2012}, \cite{belu2013}, \cite{milano}, \cite%
{algebra}).

We set 
\begin{equation*}
\mathfrak{L}=\mathcal{P}_{\omega }(\mathbb{R}),
\end{equation*}%
where $\mathcal{P}_{\omega }(\mathbb{R}^{N})$ denotes the family of finite
subsets of $\mathbb{R}$. We will refer to $\mathfrak{L}$ as the "parameter
space". Clearly $\left( \mathfrak{L},\subset \right) $ is a directed set%
\footnote{%
We recall that a directed set is a partially ordered set $(D,\prec )$ such
that, $\forall a,b\in D,\ \exists c\in D$ such that 
\begin{equation*}
a\prec c\ \ \text{and}\ \ b\prec c.
\end{equation*}%
}. A function $\varphi :D\rightarrow E$ defined on a directed set will be
called \textit{net }(with values in $E$). A net $\varphi $ is the
generalization of the notion of sequence and it has been constructed in such
a way that the Weierstrass definition of limit makes sense: if $\varphi
_{\lambda }$ is a real net, we have that 
\begin{equation*}
\underset{\lambda \rightarrow \infty }{\lim }\varphi _{\lambda }=L
\end{equation*}%
if and only if 
\begin{equation}
\forall \varepsilon >0\text{ }\exists \lambda _{0}>0\text{\ such that, }%
\forall \lambda >\lambda _{0},\ \left\vert \varphi _{\lambda }-L\right\vert
<\varepsilon .  \label{limite}
\end{equation}

The key notion of $\Lambda $-theory is the $\Lambda $-limit. Also the $%
\Lambda $-limit is defined for real nets but it differs from the Weierstrass
limit defined by (\ref{limite}) mainly for the fact that there exists a non-Archimedean field in which every real net admits a limit.

We present the notion of $\Lambda $-limit axiomatically:

\bigskip

{\Large Axioms of\ the }$\Lambda ${\Large -limit}

\begin{itemize}
\item \textsf{(}$\Lambda $-\textsf{1)}\ \textbf{Existence Axiom.}\ \textit{%
There is a superreal field} $\mathbb{K}\supset \mathbb{R}$ \textit{such that
every net }$\varphi :\mathfrak{L}\rightarrow \mathbb{R}$\textit{\ has a
unique limit }$L\in \mathbb{K}{\ }($\textit{called the} "$\Lambda $-limit" 
\textit{of}\emph{\ }$\varphi .)$ \textit{The} $\Lambda $-\textit{limit of }$%
\varphi $\textit{\ will be denoted as} 
\begin{equation*}
L=\lim_{\lambda \uparrow \Lambda }\varphi (\lambda ).
\end{equation*}%
\textit{Moreover we assume that every}\emph{\ }$\xi \in \mathbb{K}$\textit{\
is the }$\Lambda $-\textit{limit\ of some real function}\emph{\ }$\varphi :%
\mathfrak{L}\rightarrow \mathbb{R}$\emph{. }

\item ($\Lambda $-2)\ \textbf{Real numbers Axiom}. \textit{If }$\varphi
(\lambda )$\textit{\ is} \textit{eventually} \textit{constant}, \textit{%
namely} $\exists \lambda _{0}\in \mathfrak{L},r\in \mathbb{R}$ \textit{such
that} $\forall \lambda \supset \lambda _{0},\ \varphi (\lambda )=r,$ \textit{%
then}%
\begin{equation*}
\lim_{\lambda \uparrow \Lambda }\varphi (\lambda )=r.
\end{equation*}

\item ($\Lambda $-3)\ \textbf{Sum and product Axiom}.\ \textit{For all }$%
\varphi ,\psi :\mathfrak{L}\rightarrow \mathbb{R}$\emph{: }%
\begin{eqnarray*}
\lim_{\lambda \uparrow \Lambda }\varphi (\lambda )+\lim_{\lambda \uparrow
\Lambda }\psi (\lambda ) &=&\lim_{\lambda \uparrow \Lambda }\left( \varphi
(\lambda )+\psi (\lambda )\right) ; \\
\lim_{\lambda \uparrow \Lambda }\varphi (\lambda )\cdot \lim_{\lambda
\uparrow \Lambda }\psi (\lambda ) &=&\lim_{\lambda \uparrow \Lambda }\left(
\varphi (\lambda )\cdot \psi (\lambda )\right) .
\end{eqnarray*}
\end{itemize}

The proof that this set of axioms $\{$($\Lambda $-1)\textsf{,}($\Lambda $%
-2),($\Lambda $-3)$\}$ is consistent can be found e.g. in \cite{ultra} or in 
\cite{belu2013}.

\subsection{Natural extension of sets and functions}

The notion of $\Lambda $-limit can be extended to sets and functions in the
following way:

\begin{definition}
\label{limito}Let $E_{\lambda },$ $\lambda \in \mathfrak{L},$ be a family of
sets in $\mathbb{R}^{N}.$ We pose%
\begin{equation*}
\lim_{\lambda \uparrow \Lambda }\ E_{\lambda }:=\left\{ \lim_{\lambda
\uparrow \Lambda }\psi (\lambda )\ |\ \psi (\lambda )\in E_{\lambda
}\right\} .
\end{equation*}%
A set which is a $\Lambda $-\textit{limit\ is called \textbf{internal}.} In
particular if, $\forall \lambda \in \mathfrak{L,}$ $E_{\lambda }=E,$ we set $%
\lim_{\lambda \uparrow \Lambda }\ E_{\lambda }=E^{\ast },\ $namely%
\begin{equation*}
E^{\ast }:=\left\{ \lim_{\lambda \uparrow \Lambda }\psi (\lambda )\ |\ \psi
(\lambda )\in E\right\} .
\end{equation*}%
$E^{\ast }$ is called the \textbf{natural extension }of $E.$
\end{definition}

Notice that, while the $\Lambda $-limit of a sequence of numbers with constant value $r\in\mathbb{R}$ is $r$, the $\Lambda$-limit of a constant sequence of sets with value $E\subseteq\mathbb{R}$ gives a larger set,
namely $E^{\ast }$. In general, the inclusion $E\subseteq E^{\ast }$ is
proper.

This definition, combined with axiom ($\Lambda $-1$)$, entails that 
\begin{equation*}
\mathbb{K}=\mathbb{R}^{\ast }.
\end{equation*}

Given any set $E,$ we can associate to it two sets: its natural extension $%
E^{\ast }$ and the set $E^{\sigma },$ where%
\begin{equation}
E^{\sigma }=\left\{ x^{\ast }\ |\ x\in E\right\} .  \label{sigmaS}
\end{equation}

Clearly $E^{\sigma }$ is a copy of $E;$ however it might be different as a set
since, in general, $x^{\ast }\neq x.$ Moreover $E^{\sigma }\subset E^{\ast }$
since every element of $E^{\sigma }$ can be regarded as the $\Lambda $%
-limit\ of a constant sequence.

\begin{definition}
\label{limito2}Let 
\begin{equation*}
f_{\lambda }:\ E_{\lambda }\rightarrow \mathbb{R},\ \ \lambda \in \mathfrak{L%
},
\end{equation*}%
be a family of functions. We define a function%
\begin{equation*}
f:\left( \lim_{\lambda \uparrow \Lambda }\ E_{\lambda }\right) \rightarrow 
\mathbb{R}^{\ast }
\end{equation*}%
as follows: for every $\xi \in \left( \lim_{\lambda \uparrow \Lambda }\
E_{\lambda }\right) $ we pose%
\begin{equation*}
f\left( \xi \right) :=\lim_{\lambda \uparrow \Lambda }\ f_{\lambda }\left(
\psi (\lambda )\right) ,
\end{equation*}%
where $\psi (\lambda )$ is a net of numbers such that 
\begin{equation*}
\psi (\lambda )\in E_{\lambda }\ \ \text{and}\ \ \lim_{\lambda \uparrow
\Lambda }\psi (\lambda )=\xi .
\end{equation*}%
A function which is a $\Lambda $-\textit{limit\ is called \textbf{internal}.}
In particular if, $\forall \lambda \in \mathfrak{L,}$ 
\begin{equation*}
f_{\lambda }=f,\ \ \ \ f:\ E\rightarrow \mathbb{R},
\end{equation*}%
we set 
\begin{equation*}
f^{\ast }=\lim_{\lambda \uparrow \Lambda }\ f_{\lambda }.
\end{equation*}%
$f^{\ast }:E^{\ast }\rightarrow \mathbb{R}^{\ast }$ is called the \textbf{%
natural extension }of $f.$
\end{definition}

More in general, the $\Lambda $-limit can be extended to a larger family of
nets; to this aim, we recall that the superstructure on $\mathbb{R}$ is
defined as follows:

\begin{equation*}
\mathbb{U}=\dbigcup_{n=0}^{\infty }\mathbb{U}_{n}
\end{equation*}%
where $\mathbb{U}_{n}$ is defined by induction as follows:%
\begin{eqnarray*}
\mathbb{U}_{0} &=&\mathbb{R}\text{;} \\
\mathbb{U}_{n+1} &=&\mathbb{U}_{n}\cup \mathcal{P}\left( \mathbb{U}%
_{n}\right) .
\end{eqnarray*}%
Here $\mathcal{P}\left( E\right) $ denotes the power set of $E.$ Identifying
the couples with the Kuratowski pairs and the functions and the relations
with their graphs, it follows that{\ }$\mathbb{U}$ contains almost every
usual mathematical object.

We can extend the definition of the $\Lambda $-limit to any bounded net%
\footnote{%
We recall that a net $\varphi :\mathfrak{X}\rightarrow \mathbb{U}$ is
bounded if there exists $n$ such that $\forall \lambda \in \mathfrak{X}%
,\varphi (\lambda )\in \mathbb{U}_{n}$.} of mathematical objects in {$%
\mathbb{U}$}. To this aim, let us consider a net%
\begin{equation}\label{carpa}
\varphi :\mathfrak{X}\rightarrow {\mathbb{U}}_{n}.
\end{equation}%
We will define $\lim\limits_{\lambda \uparrow \Lambda }\varphi (\lambda )$
by induction on $n$. For $n=0,$ $\lim\limits_{\lambda \uparrow \Lambda
}\varphi (\lambda )$ is defined by the axioms \textsf{(}$\Lambda $-\textsf{%
1),}($\Lambda $-2),($\Lambda $-3); so by induction we may assume that the
limit is defined for $n-1$ and we define it for the net (\ref{carpa}) as
follows:%
\begin{equation}
\lim_{\lambda \uparrow \Lambda }\varphi (\lambda )=\left\{ \lim_{\lambda
\uparrow \Lambda }\psi (\lambda )\ |\ \psi :\mathfrak{X}\rightarrow \mathcal{%
\mathbb{U}}_{n-1}\text{ and}\ \forall \lambda \in \mathfrak{X},\ \psi
(\lambda )\in \varphi (\lambda )\right\} .  \label{limitu}
\end{equation}

\begin{definition}
A mathematical entity (number, set, function or relation) which is the $%
\Lambda $-limit of a net is called \textbf{internal}.
\end{definition}

Let us note that, if $\left( f_{\lambda }\right) $, $\left( E_{\lambda
}\right) $ are, respectively, a net of functions and a net of sets, the $%
\Lambda -$limit of these nets defined by $\left( \ref{limitu}\right) $
coincides with the $\Lambda -$limit given by Definitions \ref{limito} and %
\ref{limito2}. The following theorem is a fundamental tool in using the $%
\Lambda $-limit:

\begin{theorem}
\label{limit}\textbf{(Leibniz Principle)} Let $\mathcal{R}$ be a relation in 
{$\mathbb{U}$}$_{n}$ for some $n\geq 0$ and let $\varphi $,$\psi :\mathfrak{X%
}\rightarrow {\mathbb{U}}_{n}$. If 
\begin{equation*}
\forall \lambda \in \mathfrak{X},\ \varphi (\lambda )\mathcal{R}\psi
(\lambda )
\end{equation*}%
then%
\begin{equation*}
\left( \underset{\lambda \uparrow \Lambda }{\lim }\varphi (\lambda )\right) 
\mathcal{R}^{\ast }\left( \underset{\lambda \uparrow \Lambda }{\lim }\psi
(\lambda )\right) .
\end{equation*}
\end{theorem}

When $\mathcal{R}$ is $\in $ or $\mathcal{=}$ we will not use the symbol $%
\ast $ to denote their extensions, since their meaning is unaltered in
universes constructed over $\mathbb{R}^{\ast }.$ To give an example of how
Leibniz Principle can be used to prove facts about internal entities, let us
prove that if $K\subseteq \mathbb{R}$ is a compact set and $(f_{\lambda })$
is a net of continuous functions then $f=\underset{\lambda \uparrow \Lambda }%
{\lim }f_{\lambda }$ has a maximum on $K^{\ast }$. For every $\lambda $ let $%
\xi _{\lambda }$ be the maximum value attained by $f_{\lambda }$ on $K$, and
let $x_{\lambda }\in K$ be such that $f_{\lambda }(x_{\lambda })=\xi
_{\lambda }.$ For every $\lambda ,$ for every $y_{\lambda }\in K$ we have
that $f_{\lambda }(y_{\lambda })\leq f_{\lambda }(x_{\lambda }).$ By Leibniz
Principle, if we pose 
\begin{equation*}
x=\lim_{\lambda \uparrow \Lambda }x_{\lambda }
\end{equation*}%
we have that%
\begin{equation*}
\forall y\in K\text{ \ }f(y)\leq f(x),
\end{equation*}

so $\xi =\lim_{\lambda \uparrow \Lambda }\xi _{\lambda }$is the maximum of $%
f $ on $K$ and it is attained on $x.$

\section{Ultrafunctions}

\subsection{Definition of Ultrafunctions}

Let $W\subset \mathfrak{F}\left( \mathbb{R}^{N},\mathbb{R}\right) \ $be a
function vector space such that $\mathcal{D}\subseteq W\subseteq L^{2}.$

\begin{definition}
\label{approxseq} We say that $(W_{\lambda })_{\lambda \in \mathfrak{L}}$ is
an \textbf{approximating net for }$W$ if

\begin{enumerate}
\item $W_{\lambda }$ is a finite dimensional vector subspace of $W$ for
every $\lambda \in \mathfrak{L}$;

\item $\lambda _{1}\subseteq \lambda _{2}\Rightarrow W_{\lambda
_{1}}\subseteq W_{\lambda _{2}}$;

\item \label{unioni} if $Z\subset W$ is a finite dimensional vector space$\ $%
then $\exists \lambda $ such that $Z\subseteq W_{\lambda }\ $(hence $%
W=\bigcup\limits_{\lambda \in \mathfrak{L}}W_{\lambda }$).
\end{enumerate}
\end{definition}

\begin{example}
Let 
\begin{equation*}
\left\{ e_{a}\right\} _{a\in \mathbb{R}}
\end{equation*}%
be a Hamel basis\footnote{%
We recall that $\left\{ e_{a}\right\} _{a\in \mathbb{R}}$ is a Hamel basis
for $W$ if $\left\{ e_{a}\right\} _{a\in \mathbb{R}}$ is a set of linearly
indipendent elements of $W$ and every element of $W$ can be (uniquely)
written has a finite sum (with coefficients in $\mathbb{R}$) of elements of $%
\left\{ e_{a}\right\} _{a\in \mathbb{R}}.$ Since a Hamel basis of $W$ has
the continuum cardinality we can use the points of $\mathbb{R}$ as indices
for this basis.
\par
{}} of $W.$ For every $\lambda \in \mathfrak{L}$ let%
\begin{equation*}
W_{\lambda }=Span\left\{ e_{a}\ |\ a\in \lambda \right\} .
\end{equation*}%
Then $(W_{\lambda })$ is an approximating net for $W.$
\end{example}

\begin{definition}
Let $(W_{\lambda })$ be an approximating net for $W$. We call \textbf{space
of ultrafunctions} \textbf{generated by }$(W,(W_{\lambda }))$ the $\Lambda $%
-limit 
\begin{equation*}
W_{\Lambda }:=\left\{ \lim_{\lambda \uparrow \Lambda }f_{\lambda }\ |\
f_{\lambda }\in W_{\lambda }\right\} .
\end{equation*}%
In this case we will also say that the space $W_{\Lambda \text{ }}$is based
on the space $W$.
\end{definition}

So a space of ultrafunctions based on $W$ depends on the choice of an
approximating net for $W$. Nevertheless, different spaces of ultrafunctions
based on $W$ have a lot of properties in common. In what follows, $%
W_{\Lambda }$ is any space of ultrafunctions based on $W$.

Since $W_{\Lambda }\subset \left[ L^{2}\right] ^{\ast },$ we can equip $%
W_{\Lambda }$ with the following inner product:%
\begin{equation*}
\left( u,v\right) =\int_{\Omega }^{\ast }u(x)\overline{v(x)}\ dx,
\end{equation*}%
where $\int^{\ast }$ is the natural extension of the Lebesgue integral
considered as a functional%
\begin{equation*}
\int :L^{1}\rightarrow {\mathbb{R}}.
\end{equation*}%
The norm of an ultrafunction will be given by 
\begin{equation*}
\left\Vert u\right\Vert =\left( \int^{\ast }|u(x)|^{2}\ dx\right) ^{\frac{1}{%
2}}.
\end{equation*}%
So, given any vector space of functions $W$, we have the following
properties:

\begin{enumerate}
\item the ultrafunctions in $W_{\Lambda }$ are $\Lambda $-limits of nets $%
(f_{\lambda })$ of functions, with $f_{\lambda }\in W_{\lambda }$ for every $%
\lambda ;$

\item the space of ultrafunctions $W_{\Lambda }$ is a vector space of
hyperfinite dimension, since it is a $\Lambda $-limit of a net of finite
dimensional vector spaces;

\item if we identify every function $f\in W$ with the ultrafunction $f^{\ast
}=\lim_{\lambda \uparrow \Lambda }f$, then $W\subset W_{\Lambda }$;

\item $W_{\Lambda }$ has a ${\mathbb{R}}^{\ast }$-valued scalar product.
\end{enumerate}

Hence the ultrafunctions are particular internal functions 
\begin{equation*}
u:\left( {\mathbb{R}^{N}}\right) ^{\ast }\rightarrow {\mathbb{R}^{\ast }.}
\end{equation*}

\begin{remark}
For every $f\in\mathfrak{F}\left( \mathbb{R}^{N},\mathbb{R}\right)$ and for every space
of ultrafunctions $W_{\Lambda}$ based on $W$ we have that $f^{\ast }\in
W_{\Lambda}$ if and only if $f\in W.$
\end{remark}

\begin{proof} Let $f\in W.$ Then, eventually, $f\in W_{\lambda }$ and
hence 
\begin{equation*}
f^{\ast }=\lim_{\lambda \uparrow \Lambda }f\in \lim_{\lambda \uparrow
\Lambda }\ W_{\lambda }=W_{\Lambda }.
\end{equation*}

Conversely, if $f\notin W$ then by the Theorem \ref{limit} it follows that $%
f^{\ast }\notin W^{\ast }$and, since $W_{\Lambda }\subset W^{\ast }$,  this
entails the thesis. 
\end{proof}

\subsection{The canonical ultrafunctions}

In this section we will introduce a space $V$ such that, given any approximating net $(V_{\lambda})$ of $V$, the space of ultrafunction $V_{\Lambda }$ generated by $(V,V_{\lambda })$ is adequate for many applications,
particularly to PDEs. The space $V$ will be called the canonical space.

Let us recall the following standard terminology: for every function $f\in
L_{loc}^{1}(\mathbb{R}^{N})$ we say that a point $x\in \mathbb{R}^{N}$ is a Lebesgue point
for $f$ if%
\begin{equation*}
f(x)=\lim_{r\rightarrow 0^{+}}\frac{1}{m(B_{r}(x))}\int_{B_{r}(x)}f(y)dy,
\end{equation*}%
where $m(B_{r}(x))$ is the Lebesgue measure of the ball $B_{r}(x);$ we
recall the very important Lebesgue differentiation theorem (see e.g. \cite%
{lebesgue}), that we will need in the following:

\begin{theorem}
\label{Lebesgue}If $f\in L_{loc}^{1}(\mathbb{R}^{N})$ then a.e. $x\in 
\mathbb{R}^{N}$ is a Lebesgue point for $f$.
\end{theorem}

We fix once for ever an infinitesimal number $\eta \neq 0.$ Given a function 
$f\in L_{loc}^{1}(\mathbb{R}^{N}),$ we set%
\begin{equation}
\overline{f}(x)=st\left( \frac{1}{m(B_{\eta }(x))}\int_{B_{\eta
}(x)}f(y)dy\right) ,  \label{due+}
\end{equation}%
where $m(B_{\eta }(x))$ is the Lebesgue measure of the ball $B_{\eta }(x).$
We will refer to the operator $f\mapsto \overline{f}$ as the Lebesgue
operator.

\begin{lemma}
\label{a}The Lebesgue operator $f\mapsto \overline{f}$ satisfies the
following properties:

\begin{enumerate}
\item \label{a1}if $x$ is a Lebesgue point for $f$ then $\overline{f}%
(x)=f(x);$

\item \label{a2}$f(x)=\overline{f}(x)\ $a.e.$;$

\item \label{a3}if $f(x)=g(x)$ a.e. then $\overline{f}(x)=\overline{g}(x);$

\item \label{a4}$\overline{\overline{f}}(x)=\overline{f}(x).$
\end{enumerate}
\end{lemma}

\begin{proof} (\ref{a1}) If $x$ is a Lebesgue point for $f$ then

\begin{equation*}
\frac{1}{m(B_{\eta }(x))}\int_{B_{\eta }(x)}f(y)dy\sim f(x),
\end{equation*}

so $\overline{f}(x)=f(x).$

(\ref{a2}) This follows immediatly by Theorem \ref{Lebesgue} and (\ref{a1}).

(\ref{a3}) Let $x\in \mathbb{R}^{N}.$ Since $f(x)=g(x)$ a.e., we obtain that $%
\int_{B_{\eta }(x)}f(y)dy=\int_{B_{\eta }(x)}g(y)dy,$ so 
\begin{eqnarray*}
\overline{f}(x) &=&st\left( \frac{1}{m(B_{\eta }(x))}\int_{B_{\eta
}(x)}f(y)dy\right) \\
&=&st\left( \frac{1}{m(B_{\eta }(x))}\int_{B_{\eta }(x)}g(y)dy\right) =%
\overline{g}(x).
\end{eqnarray*}

(\ref{a4}) This follows easily by (\ref{a2}) and (\ref{a3}). 
\end{proof}

\begin{example}
If $E=\Omega $ is an open set with smooth boundary, we have that%
\begin{equation}
\overline{\chi _{\Omega }}(x)=\left\{ 
\begin{array}{cc}
1 & \text{if\ }\ x\in \Omega ; \\ 
0 & \text{if\ }\ x\notin \Omega ; \\ 
\frac{1}{2} & \text{if\ }\ x=\partial \Omega .%
\end{array}%
\right.  \label{chi}
\end{equation}
\end{example}

We recall the following definition:

\begin{definition}
Let $f\in L^{1}(\mathbb{R}^{N}).$ $f$ is a bounded variation function (BV for short) if
there exists a finite vector Radon measure $\func{grad}f$ such that, for
every $g\in \mathcal{C}_{c}^{1}(\mathbb{R}^{N},\mathbb{R}^{N}),$ we have%
\begin{equation*}
\int f(x)\func{div}g(x)dx=-\left\langle \func{grad}f,g\right\rangle .
\end{equation*}
\end{definition}

Let us note that $\func{grad}f\ $is the gradient of $f(x)$ in the sense of
distribution. Thus, the above definition can be rephrased as follows: $f$ is a bounded variation function if $\func{grad}f\in \mathfrak{M}$.

We now set%
\begin{equation*}
V=\left\{ u\in BV_{c}(\mathbb{R}^{N})\cap L^{\infty }(\mathbb{R}^{N})\ |\ \overline{u}(x)=u(x)\right\} ,
\end{equation*}%
where $BV_{c}$ denotes the set of function of bounded variation with compact
support. So, by Lemma \ref{a},(\ref{a4}), we have that if $u\in BV_{c}(\mathbb{R}^{N})\cap
L^{\infty }(\mathbb{R}^{N})$ then $\overline{u}\in V.$ Let us observe that the condition $%
\overline{u}(x)=u(x)$ entails that the essential supremum of any $u\in V$
coincides with the supremum of $u,$ namely $\left\Vert u\right\Vert
_{L^{\infty }}=\sup \left\vert u(x)\right\vert $.

We list some properties of $V$ that will be useful in the following:

\begin{theorem}
\label{diana}The following properties hold:

\begin{enumerate}
\item \label{b1}$V\ $is a vector space and $\mathcal{C}_{c}^{1}(\mathbb{R}^{N})\subset
V\subset L^{p}(\mathbb{R}^{N})$ for every $p\in \left[ 1,+\infty \right] ;$

\item \label{b2}if $u\in V$ then the weak partial derivative $\partial _{j}u=%
\frac{\partial u}{\partial x_{j}}$ is a Radon finite signed measure;

\item \label{b3}if $u,v\in V$ then $u=v\ $a.e. if and only if $u=v$;

\item \label{b4}the $L^{2}$ norm is a norm for $V$ (and not a pseudonorm).
\end{enumerate}
\end{theorem}

\begin{proof} (\ref{b1}) This follows easily by the definitions and the fact
that $BV_{c},\ L^{\infty }$ are vector spaces.

(\ref{b2}) This holds since $V\subset BV_{c}(\mathbb{R}^{N}).$

(\ref{b3}) Let $u,v\in V$. If $u=v$ then clearly $u=v$ a.e.; conversely, let
us suppose that $u=v$ a.e.; by Theorem \ref{a}, (\ref{a3}) we deduce that $%
\overline{u}(x)=\overline{v}(x).$ But $u,v\in V,$ so $u(x)=\overline{u}(x)=%
\overline{v}(x)=v(x).$

(\ref{b4}) Let $u\in V$ be such that $\left\Vert u\right\Vert _{L^{2}}=0.$
Then $u=0$ a.e. and, since $0\in V,$ by (\ref{b3}) we deduce that $u=0.$ \end{proof}

\begin{remark}
By Theorem \ref{diana}, (\ref{b4}) it follows that, for every $f\in V,$ $%
\partial _{j}f\in V^{\prime }$ where $V^{\prime }$ denotes the (algebraic)
dual of $V$. This relation is very important to define the ultrafunction
derivative (see section \ref{D}). In fact if $f,g\in V,$ then%
\begin{equation*}
\left\langle f,\partial _{j}g\right\rangle
\end{equation*}%
is well defined, since $\partial _{j}g$ is a finite Randon measure and $f$
is a bounded Borel-measurable function and hence $f\in \mathfrak{M}^{\prime
} $.
\end{remark}

\begin{definition}
A bounded Caccioppoli set $E\subseteq\mathbb{R}^{N}$ is a Borel set such that $\chi _{E}\in
BV_{c}(\mathbb{R}^{N}),\ $namely such that $\func{grad}(\chi _{E})$ (the distributional
gradient of the characteristic function of $E$) is a finite radon measure.
The number%
\begin{equation*}
\left\langle 1,\func{grad}(\chi _{E})\right\rangle
\end{equation*}%
is called Caccioppoli perimeter of $E.$
\end{definition}

We set%
\begin{equation*}
\mathfrak{B}=\left\{ \Omega \ \text{is a bounded, open, Caccioppoli set in} \ \mathbb{R}^{N} \right\} .
\end{equation*}

Let us note $\mathfrak{B}$ is closed under unions and intersections and
that, by definition, if $\Omega \in \mathfrak{B}$ then $\overline{\chi
_{\Omega }}\in V.$

\begin{lemma}
\label{geppetto}If $f,g\in V$ and $\Omega ,\Theta \in \mathfrak{B,}$ then, 
\begin{equation*}
\overline{f\chi _{\Omega }},\ \overline{g\chi _{\Theta }}\in V;
\end{equation*}%
moreover, we have that 
\begin{equation*}
\int \overline{f\chi _{\Omega }}\ \overline{g\chi _{\Theta }}\
dx=\int_{\Omega \cap \Theta }f(x)g(x)dx.
\end{equation*}
\end{lemma}

\begin{proof} $BV_{c}(\mathbb{R}^{N})\cap L^{\infty }(\mathbb{R}^{N})$ is an algebra, so $f\chi _{\Omega }$%
,$g\chi _{\Theta }\in BV_{c}(\mathbb{R}^{N})\cap L^{\infty }(\mathbb{R}^{N})$; by Lemma \ref{a},(\ref{a3}), $%
\overline{f\chi _{\Omega }},\ \overline{g\chi _{\Theta }}\in BV_{c}(\mathbb{R}^{N})\cap
L^{\infty }(\mathbb{R}^{N})$ and by Lemma \ref{a},(\ref{a4}), $\overline{f\chi _{\Omega }},\ 
\overline{g\chi _{\Theta }}\in V.$ Using again Lemma \ref{a},(\ref{a3}), we
have that%
\begin{equation*}
\int \overline{f\chi _{\Omega }}\ \overline{g\chi _{\Theta }}=\int f\chi
_{\Omega }g\chi _{\Theta }=\int_{\Omega \cap \Theta }f\ g.\text{ \ }
\end{equation*}
\end{proof}

Now let $(V_{\lambda })$ be an approximating net for $V.$

\begin{definition}
The space of ultrafunctions $V_{\Lambda }$ generated by $(V,(V_{\lambda }))$
is called the \textbf{canonical space} of ultrafunctions, and its elements
are called \textbf{canonical ultrafunctions}.
\end{definition}

We will denote by $\mathfrak{B}_{\Lambda}$ the set 

\begin{equation*} \mathfrak{B}_{\Lambda}=\{\Omega\in\mathfrak{B}^{\ast}\mid \overline{\chi_{\Omega}}\in V_{\Lambda}\}. \end{equation*}

Let us note that, by construction, $\mathfrak{B}^{\sigma}\subseteq\mathfrak{B}_{\Lambda}\subseteq V_{\Lambda}$.

The canonical space of ultrafunctions has three important properties for
applications:

\begin{enumerate}
\item \label{i} $V_{\Lambda }\subseteq \left( L^{2}\right) ^{\ast };$

\item \label{ii} since $V_{\Lambda }\subseteq V^{\ast }$ we have that%
\begin{equation}
u\in V_{\Lambda }\Rightarrow \partial _{j}u\in V_{\Lambda }^{\prime },
\label{bona}
\end{equation}%
where $V_{\Lambda }^{\prime }$ denotes the dual of $V_{\Lambda }$;

\item \label{iii} if $\Omega $ is a bounded open set with smooth boundary,
then $\overline{\chi _{\Omega ^{\ast }}}\in V_{\Lambda }.$
\end{enumerate}

Property \ref{i} is in common with (almost) all the space of ultrafunctions
that we considered in our previous works (see \cite{ultra}, \cite{belu2012}, 
\cite{belu2013}, \cite{milano}, \cite{algebra}); it is important since it
gives a duality which corresponds to the scalar product in $L^{2}$. This
fact allows to relate the generalized solutions in the sense of
ultrafunctions with the weak solutions in the sense of distributions.

Property \ref{ii} follows by the construction of $V_{\Lambda }$, since $%
V_{\Lambda }\subseteq V^{\ast }$. This relation is used to define the
ultrafunction derivative (see section \ref{D}). There are other spaces such
as $C_{c}^{1}$, $H^{1}$ or the fractional Sobolev space $H^{1/2}$ which
satisfy (\ref{bona}); the fractional Sobolev space $H^{1/2}\ $is the optimal
Sobolev space with respect to this request (in the sense that it is the
biggest space). However our choice of the space is due to the request \ref%
{iii}. This request seems necessary to get a definite integral which
satifies the properties which allows to prove Gauss' divegence theorem (see
section \ref{DI}) and hence to prove some conservation laws. Also this
property implies that the extensions of local operators\footnote{%
By local operator we mean any operator $F:V\rightarrow V$
such that supp$(F(f))\subseteq $ supp$(f)$ $\forall f\in V.$} are
local.

Let us note that there are other spaces which satisfy \ref{i}, \ref{ii}, \ref%
{iii}, e.g the space generated by functions of the form 
\begin{equation*}
u(x)=\overline{f(x)\chi _{\Omega }(x)}
\end{equation*}%
with $f\in \mathcal{C}^{2}(\mathbb{R}^{N}).$ Clearly this space is included in $V$ and so it
seems more convenient to take $V.$ In any case, we think that $V$ is a good
framework for our work.

\subsection{Canonical extension of functions and measures}

We denote by $\mathfrak{M}$ the vector space of\textit{\ (signed) Radon
measure }on $\mathbb{R}^{N}$.

We start by defining a map%
\begin{equation*}
P_{\Lambda }:\mathfrak{M}^{\ast }\rightarrow V_{\Lambda }
\end{equation*}%
which will be very useful in the extension of functions. As usual we will
suppose that $L_{loc}^{1}(\mathbb{R}^{N})\subset \mathfrak{M}$ identifying every locally
integrable function $f$ with the measure $f(x)dx.$

\begin{definition}
\label{tilde} If $\mu \in \mathfrak{M}^{\ast },$ $\widetilde{\mu }%
=P_{\Lambda }\mu $ denotes the unique ultrafunction such that 
\begin{equation*}
\forall v\in V_{\Lambda },\ \int \widetilde{\mu }(x)v(x)dx=\left\langle
v,\mu \right\rangle .
\end{equation*}%
In particular, if $u\in \left[ L_{loc}^{1}(\mathbb{R}^{N})\right] ^{\ast },$ $\widetilde{u}%
=P_{\Lambda }u$ denotes the unique ultrafunction such that%
\begin{equation*}
\forall v\in V_{\Lambda },\ \int \widetilde{u}(x)v(x)dx=\int u(x)v(x)dx.
\end{equation*}
\end{definition}

Let us note that this definition is well posed since every ultrafunction $%
v\in V_{\Lambda }$ is $\mu $-integrable for every $\mu \in \mathfrak{M}%
^{\ast }$ and hence $v\in \left( \mathfrak{M}^{\ast }\right) ^{\prime }.$

\begin{remark}
\label{rina} Notice that, if $u\in \left[ L^{2}\left( \mathbb{R}^{N}\right) %
\right] ^{\ast },$ then $P_{\Lambda }(u)$ is the orthogonal projection of $u$ on $V_{\Lambda}$.
\end{remark}

In particular, if $f\in L_{loc}^{1}(\mathbb{R}^{N}),$ the function $\widetilde{f^{\ast }}$
is well defined. From now on we will simplify the notation just writing $%
\widetilde{f}.$

\begin{example}
\label{esempio1} Take $\frac{1}{|x|},\ x\in \mathbb{R}^{N};\ $if $N\geq 2$,
then $\frac{1}{|x|}\in L_{loc}^{1}(\mathbb{R}^{N}),$ and it is easy to check that the value
of $\widetilde{\frac{1}{|x|}}$ for $x=0$ is an infinite number. Notice that
the ultrafunction $\widetilde{\frac{1}{|x|}}$ is different from $\left( 
\frac{1}{|x|}\right) ^{\ast }$ since the latter is not defined for $x=0$.
Moreover they differ "near infinity" since $\widetilde{\frac{1}{|x|}}$ has
its support is contained in an interval (of infinite lenght).
\end{example}

\begin{example}
\label{esempio2}If $E$ is a bounded borel set, then%
\begin{equation*}
\widetilde{\chi _{E}^{\ast}}=\left( \overline{\chi _{E}}\right) ^{\ast }.
\end{equation*}
\end{example}

\section{Generalization of some basic notions of calculus}

\subsection{Derivative\label{D}}

As we already mentioned, the crucial property that we will use to define the
ultrafunctions derivative is that the weak derivative of a $BV$ function is
a Radon measure. This allows to introduce the following definition:

\begin{definition}
\label{deri}Given an ultrafunction $u\in V_{\Lambda },$ we define the 
\textbf{ultrafunction derivative} as follows:%
\begin{equation*}
D_{j}u=P_{\Lambda }(\partial _{j}u)=\widetilde{\partial _{j}u,}
\end{equation*}%
where $P_{\Lambda }$ is defined by Definition \ref{tilde}.
\end{definition}

Let us note that, by definition, the ultrafunction derivative and the classical derivative of an ultrafunction $u$ coincide whenever $\partial_{j} u$ is an ultrafunction. The above definition makes sense since $\partial _{j}u\in \mathfrak{M}^{\ast
}.$ More explicitly if $u\in V_{\Lambda }$ then, $\forall v\in V_{\Lambda },$%
\begin{equation*}
\int D_{j}uv\ dx=\left\langle v,\partial _{j}u\right\rangle .
\end{equation*}%
The right hand side makes sense since $|v|$ is bounded and $\partial _{j}u$
is a finite measure.

\begin{theorem}
The ultrafunction derivative is antisymmetric; namely, for every ultrafunctions $u,v\in
V_{\Lambda }$ we have that%
\begin{equation}
\int D_{j}u(x)v(x)dx=-\int u(x)D_{j}v(x)dx.  \label{cinzia}
\end{equation}
\end{theorem}

\begin{proof} Let us observe that $BV_{c}(\mathbb{R}^{N})\cap L^{\infty }(\mathbb{R}^{N})$ is an algebra,
so $u\cdot v\in (BV_{c}(\mathbb{R}^{N})\cap L^{\infty }(\mathbb{R}^{N}))^{\ast}.$ Let $\Omega \in \mathfrak{B}_{\Lambda}$
contain the support of $u\cdot v.$ Then 
\begin{equation*}
0=\left\langle uv,\partial _{j}\overline{\chi _{\Omega }}\right\rangle
=\left\langle \overline{\chi _{\Omega }},\partial _{j}\left( uv\right)
\right\rangle =\left\langle 1,\partial _{j}\left( uv\right) \right\rangle
\end{equation*}%
By definition,%
\begin{equation*}
\int D_{j}u(x)v(x)dx+\int u(x)D_{j}v(x)dx=\left\langle u,\partial
_{j}v\right\rangle +\left\langle v,\partial _{j}u\right\rangle
\end{equation*}%
and since $u$,$v\in V^{\ast },$ then $u\partial _{j}v$ and $v\partial _{j}u$ are
Radon measures, so we have%
\begin{equation*}
\left\langle u,\partial _{j}v\right\rangle +\left\langle v,\partial
_{j}u\right\rangle =\left\langle 1,u\partial _{j}v\right\rangle
+\left\langle 1,v\partial _{j}u\right\rangle
\end{equation*}

Then%
\begin{equation*}
\int D_{j}u(x)v(x)dx+\int u(x)D_{j}v(x)dx=\left\langle 1,u\partial
_{j}v\right\rangle +\left\langle 1,v\partial _{j}u\right\rangle
=\left\langle 1,\partial _{j}\left( uv\right) \right\rangle =0
\end{equation*}%
hence we obtain the thesis. \end{proof}

\subsection{Gauss' divergence theorem\label{DI}}

\begin{definition}
\label{andromaca}If $u\in \left( L_{loc}^{1}(\mathbb{R}^{N})\right) ^{\ast }$ and $\Omega
\in \mathfrak{B}_{\Lambda }$ then we set 
\begin{equation*}
\int_{\Omega }u\ dx:=\int u\ \overline{\chi_{\Omega }}\ dx
\end{equation*}
\end{definition}

The above definition makes sense for any internal open set (the lambda-limit
of a net of open sets) and more in general for any internal Borel set.
However, the integral extended to a set in $\mathfrak{B}_{\Lambda }$ has
nicer properties, as it will be shown below. For example if $u$ and $\Omega $
are standard, the above integral concides with the usual one.

In the following we want to deal with some classical theorem in field theory such as
Gauss' divergence theorem. To do this we need some new notations. The
gradient and the divergence of a standard function (distribution) or of an
internal function (distribution) will be denoted by%
\begin{equation*}
\func{grad},\ \func{div}
\end{equation*}%
respectively; their generalization to ultrafunctions will be denoted by:%
\begin{equation*}
\nabla ,\ \nabla \cdot .
\end{equation*}%
Namely, if $u\in V_{\Lambda },$ we have that%
\begin{equation*}
\func{grad}u=\left( \partial _{1}u,...,\partial _{N}u\right) ;\ \ \nabla
u=\left( D_{1}u,....,D_{N}u\right) .
\end{equation*}%
Similarly, if $\phi =\left( \phi _{1},...,\phi _{N}\right) \in \left(
V_{\Lambda }\right) ^{N},$ we have that%
\begin{equation*}
\func{div}\phi =\partial _{1}\phi _{1}+....+\partial _{N}\phi _{N};\ \
\nabla \cdot \phi =D_{1}\phi _{1}+....+D_{N}\phi _{N}.
\end{equation*}

If $\Omega \in \mathfrak{B,}$ then $\func{grad}\chi _{\Omega }=\left(
\partial _{1}\chi _{\Omega },...,\partial _{N}\chi _{\Omega }\right) $ is a
vector-valued Radon measure such that, $\forall \phi \in \left( \mathcal{C}%
^{1}(\mathbb{R}^{N})\right) ^{N},$%
\begin{equation}
\left\langle \func{grad}\chi _{\Omega },\phi \right\rangle =-\int_{\Omega }%
\func{div}\phi \ dx  \label{palla}
\end{equation}

As usual, we will denote by $\left\vert \func{grad}\chi _{\Omega
}\right\vert $ the total variation of $\chi _{\Omega },$ namely a Radon
measure defined as follows: for any Borel set $A,$ 
\begin{equation*}
\left\vert \func{grad}\chi _{\Omega }\right\vert (A)=\sup \left\{
\left\langle \func{grad}\chi _{\Omega },\phi \right\rangle \ |\ \phi \in
\left( \mathcal{C}^{1}(\mathbb{R}^{N})\right) ^{N},\ \left\vert \phi (x)\right\vert \leq
1\right\} .
\end{equation*}%
$\left\vert \func{grad}\chi _{\Omega }\right\vert $ is a measure
concentrated on $\partial \Omega $ and the quantity 
\begin{equation*}
\left\langle \func{grad}\chi _{\Omega },1\right\rangle 
\end{equation*}%
is called Caccioppoli perimeter of $\Omega $ (see e.g. \cite{cac})$.$ If $%
\partial \Omega $ is smooth, then $\left\vert \func{grad}\chi _{\Omega
}\right\vert $ agrees with the usual surface measure and hence if $f$ is a
Borel function, $\left\langle f,\left\vert \func{grad}\chi _{\Omega
}\right\vert \right\rangle $ is a generalization of the surface integral $%
\int_{\partial \Omega }f(x)d\sigma .$ This generalization suggests a further
generalization in the framework of ultrafunctions:

\begin{definition}
\label{surf}If $u\in V^{\ast }$ and $\Omega \in \mathfrak{B}_{\Lambda }$
then we set 
\begin{equation*}
\int_{\partial \Omega }u\ d\sigma :=\int u\ \left\vert \nabla \overline{\chi _{\Omega
}}\right\vert \ dx,
\end{equation*}%
where $\left\vert \nabla \overline{\chi _{\Omega }}\right\vert $ is the ultrafunction
defined by the following formula: $\forall v\in V_{\Lambda }$%
\begin{equation*}
\int \left\vert \nabla \overline{\chi _{\Omega }}\right\vert v(x)dx=\left\langle
v,\left\vert \func{grad}\overline{\chi _{\Omega }}\right\vert \right\rangle .
\end{equation*}
\end{definition}

\begin{lemma}
\label{G1} If $\phi \in \left( V_{\Lambda }\right) ^{N}$ and $\Omega \in 
\mathfrak{B}_{\Lambda}$ then%
\begin{equation}
\int_{\Omega }\nabla \cdot \phi \ dx=-\int \phi \cdot \nabla \overline{\chi
_{\Omega }}\ dx.  \label{gauss}
\end{equation}
\end{lemma}

\begin{proof} We have that%
\begin{eqnarray*}
\int_{\Omega }\nabla \cdot \phi \ dx &=&\sum_{j}\int_{\Omega }D_{j}\phi
_{j}\ dx=\sum_{j}\int D_{j}\phi _{j}\overline{\chi _{\Omega }}\ dx \\
&=&-\sum_{j}\int \phi _{j}D_{j}\overline{\chi _{\Omega }}\ dx=-\int \phi
\cdot \nabla \overline{\chi _{\Omega }}\ dx.
\end{eqnarray*}
\end{proof}

We can give to the Gauss theorem a more meaningful form: let%
\begin{equation*}
\nu _{\Omega }(x)=\left\{ 
\begin{array}{cc}
-\frac{\nabla \overline{\chi _{\Omega }}(x)}{\left\vert \nabla \overline{\chi _{\Omega
}}(x)\right\vert } & \text{if}\ \ \left\vert \nabla \chi _{\Omega
}(x)\right\vert \neq 0; \\ 
0 & \text{if}\ \ \left\vert \nabla \chi _{\Omega }(x)\right\vert =0.%
\end{array}%
\right.
\end{equation*}

Let us note that, by construction, $\nu _{\Omega }(x)$ is an internal
function whose support is infinitely close to $\partial\Omega.$

\begin{theorem}
(\textbf{Gauss' divergence theorem for ultrafunctions}) If $\phi \in \left(
V_{\Lambda }\right) ^{N}$ and $\Omega \in \mathfrak{B}_{\Lambda }$ then%
\begin{equation}
\int_{\Omega }\nabla \cdot \phi \ dx=\int_{\partial \Omega }\phi \cdot \nu
_{\Omega }(x)\ d\sigma .
\end{equation}
\end{theorem}

\begin{proof} We have that $\nabla \overline{\chi _{\Omega }}=-\nu _{\Omega
}\left\vert \nabla \overline{\chi _{\Omega }}\right\vert $ and, by using Lemma \ref%
{G1} and Definition \ref{surf}, we get:%
\begin{eqnarray*}
\int_{\Omega }\nabla \cdot \phi \ dx &=&-\int \phi \cdot \nabla \overline{\chi
_{\Omega }}\ dx= \\
\int \phi \cdot \nu _{\Omega }\ \left\vert \nabla \overline{\chi
_{\Omega }}\right\vert \ dx &=&\int_{\partial \Omega }\phi \cdot \nu _{\Omega }\ d\sigma .
\end{eqnarray*}
\end{proof}

\subsection{A simple application\label{asa}}

Let us consider the following Cauchy problem:%
\begin{eqnarray}
\frac{\partial u}{\partial t}+\func{div}F(t,x,u) &=&0;  \label{PC} \\
u(0,x) &=&u_{0}(x),  \notag
\end{eqnarray}%
where $x\in \mathbb{R}^{N}.$ It is well known that this problem has no
classical solutions since it develops singularities.

One way to formulate this problem in the framework of ultrafunctions is the
following:%
\begin{equation*}
\text{find\ }u\in \mathcal{C}^{1}(\left[ 0,T\right] ,V_{\Lambda })\text{
such that:}
\end{equation*}%
\begin{equation}
\forall v\in V_{\Lambda },\ \int \left[ \partial _{t}u+\nabla \cdot F(t,x,u)%
\right] v(x)dx=0;  \label{PT}
\end{equation}%
\begin{equation*}
u(0,x)=u_{0}^{\ast }(x)
\end{equation*}%
where $\partial _{t}=\left( \frac{\partial }{\partial t}\right) ^{\ast }$
and $u_{0}\in \mathcal{C}_{c}^{1}.$

We assume that 
\begin{equation}
F\in \mathcal{C}^{1}  \label{AA}
\end{equation}
and that%
\begin{equation}
\left\vert F(t,x,u)\right\vert \leq c_{1}+c_{2}|u|.  \label{BB}
\end{equation}

\begin{theorem}
Problem (\ref{PT}) has a unique solution and it satisfies the following
conservation law:%
\begin{equation}
\partial _{t}\int_{\Omega }u(t,x)\ dx=-\int_{\partial \Omega
}F(t,x,u(t,x))\cdot \nu _{\Omega }(x)\ d\sigma  \label{CC}
\end{equation}%
for every $\Omega \in \mathfrak{B}_{\Lambda }.$ In particular if $F(t,x,0)=0$
for every $(t,x)\in \lbrack 0,T]\times \mathbb{R}^{N},$ then 
\begin{equation}
\partial _{t}\int u(t,x)\ dx=0.  \label{DD}
\end{equation}
\end{theorem}

\begin{proof} First let us prove the existence. For every $\lambda \in 
\mathfrak{L}$, let us consider the problem%
\begin{equation*}
\text{find}\ u\in \mathcal{C}^{1}(\left[ 0,T\right] ,V_{\lambda })\ \text{%
such\ that:}
\end{equation*}%
\begin{eqnarray}
\partial _{t}u+P_{\lambda }\func{div}F(t,x,u) &=&0;  \label{PTL} \\
u(0,x) &=&u_{0}^{\ast }(x),  \notag
\end{eqnarray}%
where $P_{\lambda }:\mathfrak{M}\rightarrow V_{\lambda }$ is the "orthogonal
projection", namely, for every $\mu \in \mathfrak{M},$ $P_{\lambda }\mu $ is
the only element in $V_{\lambda }$ such that,%
\begin{equation*}
\forall v\in V_{\lambda },\ \int P_{\lambda }\mu \ v\ dx=\left\langle v,\mu
\right\rangle .
\end{equation*}%
In the above equation we have assumed that $\lambda $ is so large that $%
u_{0}^{\ast }(x)\in V_{\lambda }.$ Equation (\ref{PTL}) reduces to an
ordinary differential equation in a finite dimensional space and hence, by (%
\ref{AA}) and (\ref{BB}), it has a unique global solution $u_{\lambda }.$
Equation (\ref{PTL}) can be rewritten in the following equivalent form:%
\begin{equation*}
\forall v\in V_{\lambda },\ \int \left[ \partial _{t}u+\func{div}F(t,x,u)%
\right] v(x)dx=0.
\end{equation*}%
Taking the $\Lambda $-limit, we get a unique solution of (\ref{PT}).

Equation (\ref{CC}) follows, as usual, from Gauss' theorem:

\begin{eqnarray*}
\partial _{t}\int_{\Omega }u(t,x)dx &=&\int_{\Omega }\partial _{t}u(t,x)dx=%
\text{(by eq. (\ref{PT})\ with }v=\widetilde{1}\text{)} \\
-\int_{\Omega }\nabla \cdot F(t,x,u(t,x))dx &=&\int_{\partial \Omega
}-F(t,x,u(t,x))\cdot \nu _{\Omega }(x)d\sigma.
\end{eqnarray*}

In particular, if $F(t,x,0)=0,$ since $u$ has compact support, we have that $%
F(t,x,u(t,x))=0$ if $|x|\geq R$ with $R$ is sufficiently large. Then, taking 
$\Omega =B_{R},$ (\ref{DD}) follows. \end{proof}


\begin{thebibliography}{99}
\bibitem{benci99} Benci V., \textsl{An algebraic approach to nonstandard
analysis, }in: Calculus of Variations and Partial differential equations,
(G.Buttazzo, et al., eds.), Springer, Berlin (1999), 285--326.

\bibitem{ultra} Benci V., \textsl{Ultrafunctions and generalized solutions,}
in: Adv. Nonlinear Stud. 13, (2013), 461--486, arXiv:1206.2257.

\bibitem{BDN2003} Benci V., Di Nasso M., \textsl{Alpha-theory: an elementary
axiomatic for nonstandard analysis}, Expo. Math. 21, (2003), 355--386.

\bibitem{belu2012} Benci V., Luperi Baglini L., \textsl{A model problem for
ultrafunctions}$,$ in: Variational and Topological Methods: Theory,
Applications, Numerical Simulations, and Open Problems. Electron. J. Diff.
Eqns., Conference 21 (2014), pp. 11-21.

\bibitem{belu2013} Benci V., Luperi Baglini L., \textsl{Basic properties of
ultrafunctions,} to appear in the WNDE2012 Conference Proceedings,
arXiv:1302.7156.

\bibitem{milano} Benci V., Luperi Baglini L., \textsl{Ultrafunctions and
applications}, to appear on DCDS-S (Vol. 7, No. 4) August 2014,
arXiv:1405.4152.

\bibitem{algebra} Benci V., Luperi Baglini L., \textsl{A non archimedean
algebra and the Schwartz impossibility theorem,}, Monatsh. Math. (2014), DOI
10.1007/s00605-014-0647-x.

\bibitem{beyond} Benci V., Luperi Baglini L., \textsl{Generalized functions
beyond distributions}, to appear on AJOM (2014), arXiv:1401.5270.

\bibitem{BLta} Benci V., Luperi Baglini L., \textsl{Generalized solutions of
the Burgers' equations}, in preparation.

\bibitem{cac} Caccioppoli R., \textsl{Sulla quadratura delle superfici piane
e curve}, Atti Accad. Naz. Lincei Cl. Sci. Fis. Mat. Natur. Rend. Lincei
(6), (1927), 142--146.

\bibitem{lebesgue} Lebesgue H.,(1910). \textsl{Sur l'int\'{e}gration des
fonctions discontinues,}, Ann. Scientifiques \'{E}c. Norm. Sup. (27),
(1910), 361--450.

\bibitem{keisler76} Keisler H.J., \textsl{Foundations of Infinitesimal
Calculus}, Prindle, Weber \& Schmidt, Boston, (1976).

\bibitem{rob} Robinson A., \textsl{Non-standard Analysis,}\ Proceedings of
the Royal Academy of Sciences, Amsterdam (Series A) 64, (1961), 432--440.
\end{thebibliography}
\end{document}